\documentclass[11pt]{amsart}
\usepackage[mathscr]{eucal}
\usepackage{amsfonts}
\usepackage{amsmath}
\usepackage{amsthm}
\usepackage{amssymb}
\usepackage{latexsym}
\usepackage{psfig}

\theoremstyle{plain}
\newtheorem{definition}[equation]{Definition}

\newtheorem{corollary}[equation]{Corollary}
\newtheorem{lemma}[equation]{Lemma}
\newtheorem{prop}[equation]{Proposition}
\newtheorem{theorem}[equation]{Theorem}

\theoremstyle{definition}
\newtheorem{remark}[equation]{Remark}

\numberwithin{equation}{subsection}

\newcommand{\C}{\mathbb C}

\newcommand{\rad}{{\mathrm{rad}}}

\newcommand{\Ad}{{\mathrm{Ad}}}

\newcommand{\ZZ}{\mathbb Z}

\def \ann {{\rm Ann}}

\def\E{{\mathsf{E}}}

\def\p{{\mathfrak p}}
\def\b{{\mathfrak b}}

\begin{document}
\title
[A Lie-theoretic construction of symplectic reflection algebras]
{A Lie-theoretic construction of spherical symplectic reflection algebras}

\author{P. Etingof, S. Loktev, A. Oblomkov, L. Rybnikov }
\address{Department of Mathematics, Massachusetts Institute of
Technology, Cambridge, MA 02139, U.S.A.}
\email{etingof@math.mit.edu}
\address{Institute for Theoretical and Experimental Physics, B.Cheremushkinskaya ul., 25, Moscow 117218, Russia}
\email{loktev@itep.ru}
\address{Department of Mathematics, Princeton University, Princeton, NJ, U.S.A}
\email{oblomkov@math.princeton.edu}
\address{Institute for Theoretical and Experimental Physics, B.Cheremushkinskaya ul., 25, Moscow 117218, Russia}
\email{leo.rybnikov@gmail.com}

\maketitle{}

\def \qr {{\mathfrak A}}
\def \g {{\mathfrak g}}
\def \a {{\mathfrak a}}
\def \l {{\mathfrak l}}
\def \gl {{\mathfrak {gl}}}
\def \gs {{\mathfrak {sl}}}
\def \D {{\mathcal D}}
\def \rep {{\rm Rep}}
\def \lie {{\rm Lie}\,}
\def \fl {{\mathcal{F}\ell}}
\def \tfl {\widetilde{\fl}}
\def \ad {{\rm ad}}
\def \gn {{\bf \Gamma}_n}
\def \i {{\bf i}}
\def \s {{\bf s}}
\def \o {{\bf o}}
\def \n {{\bf n}}
\def \E {{\mathcal E}}
\def \Fun {{\mathcal Fun}}
\def \tr {{\rm Tr}}
\def \hom {{\rm Hom}}
\def \inc {{\rm Inc}}
\def \enm {{\rm End}}
\def \itw {{\mathfrak I}}
\def \fou {{\mathfrak F}}

\begin{abstract}
We propose a construction of the spherical subalgebra of a symplectic
reflection algebra of an arbitrary rank corresponding to a star-shaped affine
Dynkin diagram. Namely, it is obtained from the universal
enveloping algebra of a certain semi-simple Lie algebra
by the process of quantum Hamiltonian reduction. 
As an application, we propose a construction of finite-dimensional
representations of the spherical subalgebra. 
\end{abstract}

\section*{Introduction}

The main result of this paper is the realization of the spherical
subalgebra of the wreath product symplectic reflection algebra of
rank $n$ of types $D_4,E_6,E_7,E_8$, introduced in \cite{EG}, as a
quantum Hamiltonian reduction of the tensor product of $m$
quotients of the enveloping algebra $U({{\mathfrak{sl}}}_{n\ell})$,
where $\ell$ is 2, 3, 4, and 6, and $m$ is 4, 3, 3, and 3,
respectively. This allows one to define a functor which attaches
a representation of the spherical symplectic reflection algebra
to a collection of $m$ representations of ${{\mathfrak
{sl}}}_{n\ell}$, which are annihilated by certain ideals. 
In particular, this gives an explicit Lie-theoretic construction of many finite
dimensional representations of spherical symplectic reflection
algebras, most of which appear to be new. In the rank 1 case, all
finite dimensional representations of spherical symplectic
reflection algebras are classified by Crawley-Boevey and Holland
(\cite{CBH}) and our construction yields several explicit
Lie-theoretic realizations of all of them.

The proof of the main result is based on the previous work \cite{EGGO},
in which the spherical subalgebra of a wreath product symplectic
reflection algebra (of any type) is realized as the quantum
Hamiltonian reduction from the algebra of differential
operators on representations (with a certain dimension vector) 
of the Calogero-Moser quiver, obtained from the corresponding 
extended Dynkin quiver by adjoining an auxiliary vertex, linked 
to the extending vertex. Namely, in the case when the extended 
Dynkin graph is star-shaped  
(which happens in the cases $D_4,E_6,E_7,E_8$), this reduction can be
performed in 2 steps: first the reduction with respect to the
groups of basis changes at the non-central vertices, and then 
with respect to the group of basis changes at the central 
(branching) vertex of the star. By the localization theorem for
partial flag varieties, after the first step, we obtain a tensor
product of $m$ quotients of enveloping algebras (the factors 
correspond to the branches of the extended Dynkin diagram), which
yields the result. 

Recall that by the results of \cite{EG}, a wreath product spherical symplectic
reflection algebra can be viewed as a quantization
of the wreath product Calogero-Moser space, which is a
deformation of the $n$-th Hilbert scheme of the resolution of a
Kleinian singularity. Our main result is a quantum analog of the 
statement from classical symplectic geometry, stating that 
this wreath product Calogero-Moser space may be obtained by
classical Hamiltonian reduction from a product of $m$ coadjoint
orbits of the Lie algebra ${{\mathfrak{sl}}}_{n\ell}$, or,
equivalently, as the space of solutions of a special kind of the additive
Deligne-Simpson problem. This classical result can be found in
\cite{EGO}, Section 2.6. 

We expect that the main result of this paper has a q-deformed analog, in which
the Lie algebra ${{\mathfrak{sl}}}_{n\ell}$ is replaced by the
corresponding quantum group $U_q({{\mathfrak{sl}}}_{n\ell})$,
and spherical symplectic reflection algebras are replaced by
spherical subalgebras of generalized double affine Hecke algebras
(DAHAs), introduced for higher rank in \cite{EGO} 
and for rank 1 in \cite{EOR}
(except the $D_4$ case, where they were known before
due to the work of Sahi and Stokman, \cite{Sa,St}).
Such a result would be a quantization of the result 
of Section 5.2 of \cite{EGO}, which is a multiplicative 
version of the abovementioned result from Section 2.6,
as it is concerned with the multiplicative, rather than additive,
Deligne-Simpson problem. In fact, we can construct 
the corresponding functor, which attaches a representation 
of the spherical subalgebra 
of the generalized DAHA to a collection of $m$ representations 
of the quantum group annihilated by certain ideals. This, as well as
non-spherical versions of the results of this paper, will be  
discussed in future publications. 

The paper is organized as follows. 
In Section 1, we introduce the notation and formulate the main result
(Theorem~\ref{thmain}). We prove it in Section 2. 
In Section 3, we use our construction to produce finite
dimensional representations of spherical symplectic reflection
algebras, starting from known classes of Lie algebra representations.

{\bf Acknowledgements.} It is our pleasure to dedicate
this paper to the 80-th birthday of Bertram Kostant;
to a significant extent, its results rest on the ideas
stemming from the orbit method. We are grateful to
A. Braverman, B.~Feigin, V.~Ginzburg, I.~Gordon, D.~Kazhdan,
A.~Levin, M.~Olshanetsky, E.~Rains, and D.~Vogan for very useful and
stimulating discussions. The work of the authors was partially
supported by CRDF grant RM1-2545-MO-03. The work of P.E. was
partially supported by the NSF grant DMS-0504847.
The work of S. L. was
supported by the RF President Grant N.Sh-8004.2006.2,
grant RFBR-08-02-00287 and the P.Deligne
scholarship.
The work of L. R. was supported by the grant RFBR-CNRS-07-01-92214
and the P.Deligne
scholarship.

\section{The Main Theorem}

\subsection{Quantum Hamiltonian Reduction}

The following construction is the quantum analog of the Hamiltonian 
reduction procedure.

Let $A$ be an associative algebra, $\g$ be a Lie algebra, $\mu:
\g\to A$ be a homomorphism of Lie algebras.

\begin{definition}
Define the associative algebra
$$\qr (A, \g, \mu) = \left(A / A\mu(\g)\right) ^{
\g},$$
where the invariants are taken with respect to the adjoint action
of $\g$ on $A$. 
This algebra is called the {\bf quantum Hamiltonian reduction}
of $A$ with respect to $\g$ with quantum moment map $\mu$.
\end{definition}

The following proposition is well known, but we give its proof
for reader's convenience. It summarizes the main properties of
the quantum Hamiltonian reduction, and gives a construction of its
representations. 

\begin{prop}\label{seqred}
Assume that $\g$ is reductive, and 
the adjoint action of $\g$ on $A$ is completely reducible. 
Then: 

(i)
$$
\qr (A, \g, \mu) = A^{\g} / (A\mu(\g))^{\g}=
A^{\g} / (\mu(\g)A)^{\g}=
(A/\mu(\g) A)^{\g}.
$$

(ii)
If $V$ is any $A$-module then 
$\qr (A, \g, \mu)$ acts naturally 
on the cohomology $H^i(\g,V)$ and the homology 
$H_i(\g,V)$, in particular, on the 
invariants $V^\g$ and the coinvariants $V_\g$. 

(iii) Suppose that $\g = \g_1 \oplus \g_2$, so $\mu = \mu_1
\oplus \mu_2$.  Then $\mu_2$ descends to a map
$\overline{\mu_2}:\g_2\to \qr (A, \g_1, \mu_1)$ and we have
$$\qr (A, \g, \mu) \cong \qr\left( \qr (A, \g_1, \mu_1), \g_2,
\overline{\mu_2}\right).$$

\end{prop}

\begin{proof} Let us prove (i). 
Since the functor of taking $\g$-invariants is
exact on completely reducible modules, the only equality we 
need to prove is (using a shorthand notation)
$(A\mu(\g))^\g=(\mu(\g)A)^\g$. 

Since $\g$ is reductive and $A$ is
completely reducible as a $\g$-module, 
the space $(A\mu(\g))^\g$ is the image of
$(A\otimes \g)^\g$ under the multiplication map. Therefore, any
element $x\in (A\mu(\g))^\g$ is the image of some element $x'=
\sum a_i\otimes b_i\in (A\otimes\g)^\g$, where $b_i$ is a basis 
of $\g$. Let $[b_ib_j]=\sum c_{ij}^kb_k$. Then 
$$
\sum [b_ja_i]\otimes b_i+\sum a_i\otimes [b_jb_i]=0,
$$
which implies that $[b_ja_i]=\sum c_{pj}^ia_p$. 
Thus $\sum_i [b_ia_i]=\sum_{p,i} c_{pi}^ia_p=0$, as 
$\sum c_{pi}^i={\rm Tr}({\rm ad}(b_p))=0$ (because $\g$ is reductive). 
So $x=\sum a_ib_i=\sum b_ia_i\in (\mu(\g)A)^\g$, as desired.  

To prove (ii), note that $A/A\mu(\g)$ obviously acts by operators
from $V^\g$ to $V$, so $(A/A\mu(\g))^\g$ acts on $V^\g$. 
Also, it is clear that $A/\mu(\g)A$ acts from $V$ to $V_\g$, so 
$(A/\mu(\g)A)^\g$ acts on $V_\g$. These actions are functorial in $V$,
so they extend to derived functors, as desired.  

Statement (iii) follows easily from the fact that 
$\g$ is reductive and acts completely reducibly on $A$. 
\end{proof} 

\subsection{Symplectic Reflection Algebras}

Let $L$
be a 2-dimensional complex vector space equipped with a symplectic
form $\omega$. Let $\Gamma$ be a finite subgroup of 
$Sp(L)\cong SL_2(\Bbb C)$. Let $\gn = S_n\ltimes \Gamma^n$, where
$S_n$ is the symmetric group.

Let $\C[\Gamma]$ be the group algebra of $\Gamma$, and 
$Z\Gamma$ be the center of $\C[\Gamma]$.
Let $c:\Gamma\setminus\lbrace{1\rbrace}\to \C$ be a conjugation
invariant function. 

For $u \in L$ denote by $u_l \in L^n$ the corresponding element in
the $l$-th summand. Similarly, for $\gamma \in \Gamma$ let $\gamma_l$ be
the element of $\Gamma^n \subset \gn$, in which $\gamma$ stands
in the $l$-th place. Let $s_{lm} \in S_n$ be the transposition of $l$ and $m$.
Let $k,t\in \Bbb C$. 

\begin{definition}
The symplectic reflection algebra 
$H_{t,k,c}=H_{t,k,c}(\gn)$ of rank $n$ associated to
$\Gamma$ is the quotient of the smash product $\C [\gn]\ltimes 
T(L^n)$ (where the group algebra $\C [\gn]$ acts on $T(L^n)$
in the obvious way), by the additional relations
$$
[u_l, v_m] = - \frac{k}{2} \sum_{\gamma \in \Gamma} \omega(\gamma u,v)
s_{lm}\gamma_l\gamma_m^{-1},
\qquad u,v \in L, \ \ l \ne m;
$$
$$
[u_l,v_l] = \omega(u,v)\left(t+\sum_{\gamma \in
\Gamma,\gamma\ne 1} c_\gamma \gamma_l +
\frac{k}{2} \sum_{m \ne l} \sum_{\gamma \in \Gamma} s_{lm}\gamma_l\gamma_m^{-1}
\right).
$$
\end{definition}

Note that for any $a \ne 0$ 
we have  $H_{at,ak,ac} \cong  H_{t,k,c}$, so there are two
essentially different cases: the {\em classical} case $t=0$ and the 
{\em quantum} case $t=1$. In this paper, we focus on the quantum
case $t=1$. 

\subsection{Enveloping algebra quotients} 

Let $\g$ be a reductive Lie algebra. 
Let $\p$ be a parabolic subalgebra in $\g$, and let $\mu : \p \to \C$
be a Lie algebra character. 
Let ${\rm Ker} \mu\subset U(\p)\subset U(\g)$ be the kernel of the
1-dimensional representation $\mu : U(\p) \to \C$, and 
$I({\rm Ker}\mu)$ be the largest two-sided ideal in $U(\g)$ contained in the
left ideal $U(\g)\cdot {\rm Ker} \mu$.
Define the algebra
$$
U_{\mu}^\p (\g) = U(\g) / I({\rm Ker} \mu).
$$

Let $M_\mu^\p={\rm Ind}_\p^\g \mu$ be the generalized Verma
module. It is easy to show that the ideal $I({\rm Ker}\mu)$
coincides with the annihilator 
$\ann(M_\mu^\p)$. Therefore, the algebra $U_{\mu}^\p (\g)$ 
acts on any representation of $\g$
with highest weight $\mu$ with respect to $\p$.

\subsection{The Main Theorem}

Let $D$ be a graph with a shape of a star, that is, a tree with one $m$-valent
vertex $\n$, called the {\em node} or the {\em branching vertex},
and the rest of the vertices 2- and 1-valent (which form $m$ ``legs''
growing from the node). We label the legs of $D$ by numbers
$1,...,m$, and let
$d_i$ be the number of vertices in the $i$-th leg including the node. Suppose that $d_1 \le \dots \le d_m$.
Then let us enumerate
the vertices (excluding $\n$) by pairs $(j,i)$, where $1 \le j
\le m$ is the leg number,
and $1 \le i < d_j$ is the number of our vertex on the leg
starting from the outside (so
$(j,d_j-1)$ is connected with $\n$).

Recall that via the McKay correspondence, $\Gamma$ corresponds 
to a finite ADE Dynkin diagram $D^0_\Gamma$, 
and the irreducible finite-dimensional
representations of $\Gamma$ correspond to vertices of the
affinization $D_\Gamma$ of $D^0_\Gamma$.
Suppose that  $D_\Gamma$ has the shape of a star. It means
that $D_\Gamma^0$ is the Dynkin graph of $D_4$ or $E_6$ or $E_7$
or $E_8$. Then the sets $(d_1,...,d_m)$ in the four cases under
consideration are: (2,2,2,2), (3,3,3), (2,4,4), and (2,3,6). Let
$\ell = d_m$, and note that in our setting all $d_i$ are divisors
of $\ell$. Denote by $\o = (m,1)$ the affinizing vertex, that is,
the vertex belonging to $D_\Gamma$ but not $D_\Gamma^0$.

\bigskip

\psfig{file=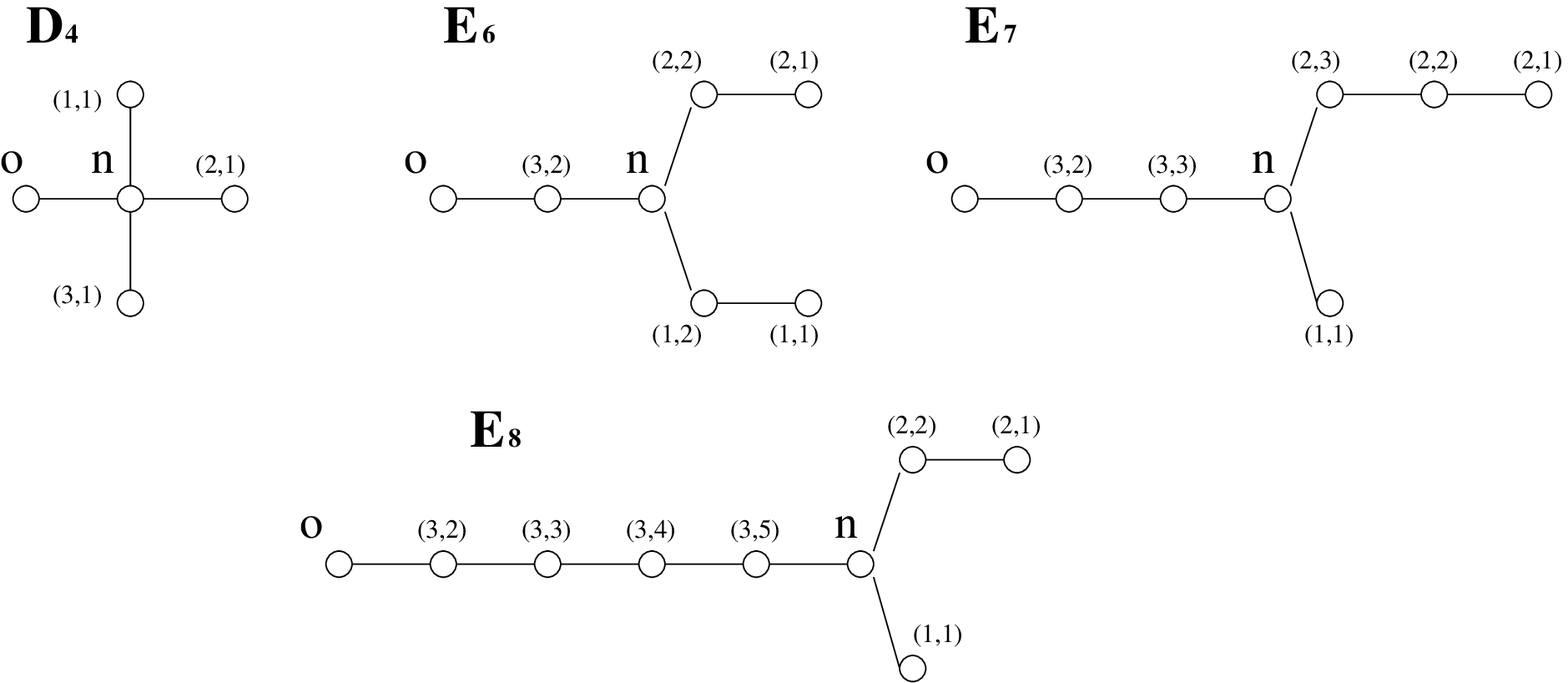,width=12.5cm}

\medskip

For a parabolic subalgebra $\p \subset \g$ and a character $\mu$ of $\p$ set
$$ \iota^\p_\mu : \g \hookrightarrow U(\g) \to U^\p_\mu(\g).$$
Denote by $\iota^{\p_1, \dots, \p_m}_{\mu_1, \dots , \mu_m}$ the natural map
$$
\sum_{i=1}^m 1^{\otimes i-1} \otimes \iota^{\p_i}_{\mu_i}
\otimes 1^{\otimes m-i}: \g\to U_{\mu_1}^{\p_1}(\g) 
\otimes \dots \otimes U_{\mu_m}^{\p_m}(\g).
$$

Let $r>1$ be a positive integer. 
For any positive integer $s$ dividing $r$, 
denote by $\p(s,r) \subset \gs_r$ the
parabolic subalgebra generated by the subalgebra 
$\b$ of upper triangular matrices, and
the elements $f_i=E_{i+1,i}$ for $i$ not divisible by 
$r/s$. Also, denote by $\p'(s,r)\subset \p(s,r)\subset \gs_r$ 
the parabolic subalgebra generated by $\b$ and $f_i$ for
$i>1$ not divisible by $r/s$.

Denote by $P_\Gamma$ the vector space spanned by the roots of the affine
root system corresponding to $\Gamma$. For each
vertex $\i$ of $D_\Gamma$ denote by $\alpha_\i\in P_\Gamma$ 
the corresponding simple root.
For a vertex $\i\in D_\Gamma$, denote by $N_\i$ the corresponding
representation of $\Gamma$.

For a weight $\lambda \in P_\Gamma$ set
$$\mu_j(n,\lambda) = \sum_{i=1}^{d_j-1} \left(\lambda_{(j,i)} - \frac{n\ell}{d_j}\right) \omega_{{n\frac{\ell}{d_j}i}},$$
where $\omega_i$ is the $i$-th fundamental weight of $\gs_{n\ell}$.

Let $c:\Gamma\setminus \lbrace{1\rbrace}$ be a conjugation
invariant function. 
Define the weight $\lambda(c)\in P_\Gamma$ by
$$
\lambda(c) = \sum_{\i} \lambda(c)_\i \alpha_\i,
$$
where 
$$
\lambda(c)_\i = \frac{1}{|\Gamma|}\left(\dim N_\i +
\sum_{\gamma\ne 1}c_\gamma\tr_{N_\i}(\gamma)\right).
$$

\begin{theorem}\label{thmain}
Let
$$\p_i = \p(d_i, n\ell), \quad  \mu_i = \mu_i(n,\lambda(c)) \quad \mbox{for} \ 1\le i <m,$$
$$\p_m = \p'(d_m, n\ell), \quad
\mu_m = \mu_m(n,\lambda(c)) + n\left(\frac{k}{2}-1\right)\omega_1 - \frac{k}{2}\omega_n.$$
Then we have
$$eH_{1,k,c}e \cong \qr 
\left(U_{\mu_1}^{\p_1}(\gs_{n\ell}) 
\otimes \dots \otimes U_{\mu_m}^{\p_m}(\gs_{n\ell}), \gs_{n\ell},
\iota^{\p_1, \dots,  \p_m}_{\mu_1, \dots, \mu_m} \right).$$
\end{theorem}

This theorem is a quantum analog of the result of Section 2.6 in
\cite{EGO}. We prove it in the next section.

\begin{corollary}\label{consrep}
Let $Y$ be a representation of $\otimes_{i=1}^m 
U_{\mu_i}^{\p_i}(\gs_{n\ell})$ (for example, 
$Y=U_1\otimes \dots \otimes  U_m$, where 
$U_i$ are representations of $U_{\mu_i}^{\p_i}(\gs_{n\ell})$, 
$i=1,...,m$). Then $eH_{1,k,c}e$ acts on
$Y^{\gs_{n\ell}}$ 
and $Y_{\gs_{n\ell}}$.
More generally, it acts on the cohomology $H^i(\gs_{n\ell},Y)$ 
and the homology 
$H_i(\gs_{n\ell},Y)$.
\end{corollary}

\begin{proof}
The corollary follows immediately 
from Theorem~\ref{thmain} and Proposition \ref{seqred}. 
\end{proof}

\section{Proof of Theorem~\ref{thmain}}

\subsection{Twisted differential operators}

One of the important applications of the quantum Hamiltonian
reduction construction is the construction of the sheaf of differential
operators on an algebraic variety, twisted by a collection of line bundles.  

Namely, let $X$ be a smooth affine complex algebraic variety, 
$\mathcal{D}(X)$ be the algebra of differential operators on $X$, 
and let $L_1, \dots L_d$ be line bundles on $X$. Let $E$
denote the total space of the principal $(\Bbb C^\times)^d$-bundle 
corresponding to $L_1 \oplus \dots \oplus L_d$. Denote by 
$\E_i$ the Euler vector field on $E$ along the the $i$-th factor 
of the fiber, and by $\E$ the map $\C^d\to {\rm Vect}(E)$ sending
the standard basis vectors to $\E_i$. 

Let $\chi \in \left(\C^d\right)^*$.
 
\begin{definition}\label{dftd}
The algebra of twisted differential operators $D_{\chi,L_1,...,L_d}(X)$ is
defined by 
$$
\D_{\chi,L_1,...,L_d}(X):=\qr (\D(E), \C^d, \E-\chi),
$$
\end{definition}

This is a flat $d$-parametric deformation of the algebra 
$\D(X) \cong \D_{0,L_1,...,L_d}(X)$.

\begin{remark}
Let $\nabla_i$ be connections on $L_i$ with curvatures $F_i$. 
Using these connections, we can naturally lift any vector 
field $v$ on $X$ to a $(\Bbb C^\times)^d$-invariant vector 
field $\nabla(v)$ on $E$, and we have 
\begin{equation}\label{curv}
[\nabla(v),\nabla(w)]=\nabla([v,w])+\sum_{i=1}^d \chi_i F_i(v,w)
\end{equation}
modulo the ideal generated by $\E_i-\chi_i$. Thus the algebra
$\D_{\chi,L_1,...,L_d}(X)$ is generated by regular
functions $f\in {\mathcal O}(X)$ and elements $\nabla(v)$, $v\in {\rm
Vect}(X)$, with relations 
$$
\nabla(fv)=f\nabla(v), [\nabla(v),f]=L_vf, 
$$
and (\ref{curv}), where $L_v$ is the Lie derivative. 
This is the usual definition of twisted differential operators.
\end{remark} 

If $X$ is smooth but not affine, then for any affine open set
$U\subset X$, we have the algebra
$$
\D_{\chi,L_1,...,L_d}(U) = \qr (\D(E(U)), \C^d, \E-\chi),
$$
where $E(U)$ is the preimage of $U$ under the map $E\to X$. 
One can show that this defines a quasicoherent sheaf 
$D_{\chi,L_1,...,L_d}$ on $X$. This sheaf is called {\it the sheaf of
twisted differential operators}. 

\begin{remark}
For more on twisted differential operators, see \cite{BB}.
\end{remark} 

\subsection{Twisted differential
operators on partial flag varieties}

Let $G$ be an algebraic group corresponding to
$\g$, and $P$ the parabolic subgroup of $G$ corresponding to
$\p$. Then to any Lie algebra character $\mu:\p\to\C$ we 
can canonically attach the algebra of global twisted 
differential operators $\D_\mu(G/P)$, where the twisting is with respect to 
the line bundles $L_i$ corresponding to some generators $\chi_i$
of the character group ${\rm Hom}(P,\C^\times)$. Namely, 
$$
\D_\mu(G/P)=\D_{\mu^1,...,\mu^s,L_1,...,L_s}(G/P),
$$
where $\mu=\sum \mu^i d\chi_i$, and $s=\dim P/[P,P]$
(it is easy to check that this construction is independent
on the choice of generators). 

It follows from the definition of twisted differential operators 
that the natural homomorphism $U(\g)\to \D_\mu(G/P)$ 
defines a homomorphism \linebreak $i_\mu: U_{\mu}^\p (\g)\to \D_\mu(G/P)$.

In general, the map $i_\mu$ is not always an isomorphism.
However, we have the following well known result. 
Recall that the algebras $U^\p_\mu (\g)$ and $\D_\mu(G/P)$ 
have natural filtrations (the first one is induced from the
filtration on $U(\g)$, and the second one is by order of differential
operators).

\begin{prop}\label{Gr-U} If $G=SL_r$ then 

(i) ${\rm gr}(\D_\mu(G/P))=\C[T^*(G/P)]$, and 
$i_\mu$ is an isomorphism of filtered algebras
for any $\mu$;

(ii) for every irreducible finite dimensional 
representation $Y$ of $G$, the multiplicity of $Y$ 
in $\D_\mu(G/P)$ equals the dimension of $Y^L$, where $L$ is the
Levi subgroup of $P$.  
\end{prop} 

This is a parabolic version of the localization theorem. 
For convenience of the reader, let us recall the proof of this
result. 

\begin{proof}
(i) Let $O_\p$ the closure of the adjoint orbit  
$\Ad(SL_r)n_\p$, where $n_\p$ is a generic
element of $\rad\ \p$. According to \cite{KP}, 
$O_\p$ is a normal variety. Moreover, we have a natural Springer map 
$\xi_0: T^*(G/P)\to O_\p$ (the moment map for the
group action), which is a resolution of singularities in this 
case (see \cite{BB}, 2.7). Indeed, each point of 
$T^*(G/P)$ can be regarded as a pair $(\p',n)$,
where $\p'$ is a parabolic subalgebra conjugate to $\p$, and
$n\in\rad\ \p'$, and the moment map sends $(\p',n)$ to $n\in
O_\p$. Thus, we just need to show that for generic $n\in O_\p$,
there is a unique $\p'$ conjugate to $\p$ 
such that $n\in\rad\ \p'$; this follows from elementary linear algebra.

Thus, the natural map $\xi^*_0: \Bbb C[O_\p]\to \Bbb
C[T^*(G/P)]$ induced by the moment map $\xi_0$ is an
isomorphism of graded algebras. Hence we can apply Theorem 5.6
from \cite{BB} (as one of the equivalent conditions stated in 
this theorem, namely, condition (iii), holds). 

 By Theorem 5.6(v) of \cite{BB}, 
the associated graded for $U^\p_\mu(\g)$ is the algebra
$\C[O_\p]$ of polynomial functions on $O_\p$. 
At the same time, the associated graded algebra of 
$\D_\mu(G/P)$ is contained in the algebra $\C[T^*(G/P)]$ of polynomial
functions on the cotangent bundle.
Also, it is clear that ${\rm gr}(i_\mu)=\xi_0^*$. 
Thus $i_\mu$ is an isomorphism of filtered algebras, 
and ${\rm gr}(\D_\mu(G/P))=\C[T^*(G/P)]$, as desired.

(ii) It suffices to show that $\Bbb C[O_\p]$ has the stated
property. There exists a 1-parameter family of semisimple orbits
$O_\p^t, t\in \Bbb C^\times$, which degenerates 
at $t=0$ into $O_\p$. These orbits have the form $G/L$. 
Thus, as a $G$-module, $\Bbb C[O_\p]=\Bbb C[G/L]$,
and the statement follows from the Peter-Weyl theorem.   
\end{proof} 

\subsection{Relation to Quivers}

Let $Q$ be a quiver with the set of vertices $I$. For any edge $a\in
Q$ denote by $h(a)$ and $t(a)$ its head and tail, respectively.

For a dimension vector $\beta \in \ZZ_+^I$
consider the space of representations
$$\rep_\beta(Q) = \bigoplus_{a \in Q} \hom
\left(\C^{\beta_{t(a)}}, 
\C^{\beta_{h(a)}} \right).$$

This space admits a natural linear action $\itw$ of the group 
$GL(\beta) = \prod_{\i \in I} GL_{\beta_\i}$.
Denote by $\gl(\beta)$ the Lie algebra of this group;
then we have the corresponding map
$d\itw: \gl(\beta) \to \D(\rep_{\beta}(Q))$.

Now let us introduce quiver-related 
twisted differential operators. For $\chi \in \C^I$ let
$$\chi \tr = \sum_{\i \in I} \chi_\i \cdot \tr_\i: \gl(\beta) \to \C,$$
where $\tr_\i$ is the trace on $\gl_{\beta_\i}$.

\begin{definition} Let
$$ \D_\chi (Q,\beta) = \qr\left(\D\left(\rep_{\beta}(Q)\right),
\gl(\beta), d\itw - \chi\tr \right).
$$
\end{definition}

Now let $\Gamma$ be of type $D_4$, $E_6$, $E_7$, $E_8$. 
Let $D_\Gamma^{CM}$ be the graph 
obtained from the graph $D_\Gamma$ by
by adding a vertex $\s$ with one edge from $\s$
to the affinizing vertex $\o$.
Let $Q_{CM}$ be a quiver obtained by 
orienting all edges of $D_\Gamma^{CM}$ in some way, so that the 
additional edge is oriented from $\s$ to $\o$ 
(the Calogero-Moser quiver, see \cite{EGGO}). 

Let $I_{CM}=I\cup \lbrace{\s\rbrace}$ be the set of vertices of
the Calogero-Moser quiver. 
Introduce the vector  $\partial\in \C^{I}$ by
$$\partial_\i = n \left(- \delta_\i + \sum_{a \in Q_{CM}: 
t(a) = \i} \delta_{h(a)}\right),$$
where $\delta:=\sum \dim N_\i\alpha_\i$ is the basic imaginary
root. 

\begin{prop}
For the orientation of all edges towards the node 
we have $\partial_{(i,j)} = n\ell/d_i$.
\end{prop}

\begin{proof} The proof is by a direct computation. 
\end{proof} 

Define  $\chi^{CM}\in \C^{I_{CM}}$ by
$$\chi^{CM}_\s = n\left(\frac{k}{2} -1\right), \quad
\chi^{CM}_\o = \lambda(c)_\o - \partial_\o -  \frac{k}{2},
\quad \chi^{CM}_\i = \lambda(c)_\i - \partial_\i,\ \i \ne \o,\s.$$

\begin{theorem}\cite{EGGO}\label{DAHA-Q}
Take $\alpha^{CM} = \alpha_\s + n \delta$. 
Then for any orientation of the quiver 
we have an isomorphism of filtered algebras
$$ \D_{\chi^{CM}} (Q_{CM},\alpha^{CM})\cong eH_{1,k,c}e.$$
\end{theorem}

\subsection{Relation to partial flag varieties}

Let $0<r_1<r_2<...<r_s<r$ be a collection of positive integers. 
Denote by $\fl(r_1, \dots, r_s; r)$ the corresponding partial flag
variety, i.e., the configuration space of $s$ subspaces
$$V_1 \subset \dots \subset V_s \subset V_{s+1}=\C^r$$
such that $\dim V_j =  r_j$, $j\le s$.

Denote by $L_j$ the line bundle $\wedge^{r_j}(V_j)$.
These bundles can be considered as
generators of the Picard group of $\fl(r_1, \dots, r_s; r)$. Then
for a vector $\chi = (\chi_1, \dots, \chi_s)$ one can define the
algebra of twisted differential operators on
$\D_{\chi,L_1,...,L_s}(\fl(r_1,...,r_s; r))$ following
Definition~\ref{dftd}. Abbreviating the notation, we will denote this algebra 
by $\D_\chi(\fl(r_1,...,r_s; r))$.

Now let $Y=\oplus_{i=1}^s {\rm Hom}(V_i,V_{i+1})$, and
$H=\prod_{i=1}^s GL(V_i)$. 
Let $\itw$ be the natural representation of $H$ on $Y$.
Note that the action of the remaining $GL_{r}$ commutes with $\itw$.
Denote by $d\itw$ the corresponding map from  the  
Lie algebra ${\rm Lie}H$ to the algebra of differential operators on $Y$. For  
$\chi = (\chi_1, \dots, \chi_s) \in \C^s$ let
$\chi\tr = \sum \chi_i \cdot \tr_i$ be a character of this Lie algebra.

\begin{theorem}\label{Q-Gr}
Let $\chi = (\chi_1, \dots, \chi_s) \in \C^s$. Then we have a
$GL_{r}$-equivariant isomorphism of filtered algebras
$$
\qr\left(\D\left(\bigoplus_{i=1}^s \hom(V_i, V_{i+1})\right), 
\bigoplus_{i=1}^{s}\gl_{r_i}, d\itw - \chi\tr \right)
\cong  \D_{{\chi}}(\fl(r_1, \dots, r_s; r)).$$
\end{theorem}

\begin{proof}
Denote by $\inc(V_i, V_{i+1}) \subset \hom(V_i, V_{i+1})$ the open set consisting of inclusions.
We have a natural homomorphism 
$$
\D\left(\bigoplus_{i=1}^s \hom(V_i, V_{i+1})\right)\to
\D\left(\prod_{i=1}^s \inc(V_i, V_{i+1})\right),
$$
which is in fact an isomorphism, because the set of non-inclusions
has codimension $\ge 2$. 

Now, the group $H$ acts freely on 
$\prod_{i=1}^s \inc(V_i, V_{i+1})$, and the quotient is 
$\fl(r_1, \dots, r_s; r)$. Thus we have a natural 
$GL_{r}$-equivariant homomorphism 
of filtered algebras 
$$
\eta: \qr\left(\D\left(\bigoplus_{i=1}^s \hom(V_i, V_{i+1})\right), 
\bigoplus_{i=1}^{s}\gl_{r_i}, d\itw - \chi\tr \right)
\to  \D_{{\chi}}(\fl(r_1, \dots, r_s; r)).
$$
We are going to show that $\eta$ is an isomorphism. 

First let us show that $\eta$ is an epimorphism. 
For this purpose note that by Proposition \ref{Gr-U},
we have a surjective map $i: U({\mathfrak{sl}}_r)\to 
\D_{{\chi}}(\fl(r_1, \dots, r_s; r))$, and 
a map 
$$
\theta: U({\mathfrak{sl}}_r)\to 
\qr\left(\D\left(\bigoplus_{i=1}^s \hom(V_i, V_{i+1})\right), 
\bigoplus_{i=1}^{s}\gl_{r_i}, d\itw - \chi\tr \right),
$$
such that $\eta\circ \theta=i$. 
This implies that $\eta$ is surjective. 

Now let us show that $\eta$ is injective. 
For this purpose, it suffices to show that 
the associated graded map ${\rm gr}(\eta)$ of $\eta$ 
is injective. For this, it suffices to prove that 
the multiplicity $m_1(W)$ of every irreducible finite dimensional $GL_r$-module 
$W$ in the source of ${\rm gr}\eta$ is at most the multiplicity
$m_2(W)$ of this module in the target (note that by Proposition
\ref{Gr-U}(ii), $m_2(W)=\dim
W^L$, where $L$ is the Levi subgroup $\prod_i GL_{r_i-r_{i-1}}$ 
in $GL_r$). In our case, the source is the algebra of regular 
functions on the classical Hamiltonian reduction 
of the space of representations of the doubled quiver of type
$A_s$ with dimension vector $(r_1,...,r_{s+1})$ with respect to
the group $H$ (i.e., the last factor $GL_r$
is not included) with the zero value of the moment map. 
Since in this case the moment map is flat (as follows from
Theorem 1.1 in 
the paper \cite{CB}; it is important here that the last factor is
not included), $m_1(W)$ equals the multiplicity of $W$ in the 
algebra of functions on the classical Hamiltonian reduction as
above for a {\it generic} value of the moment map. 
But by a Lemma of Crawley-Boevey, 
this reduction is a semisimple coadjoint orbit
of $GL_r$, isomorphic to $GL_r/L$ (see e.g. \cite{EGO}, 
Lemma 2.6.8). This implies that $m_1(W)=m_2(W)$, and $\eta$ is an
isomorphism, as desired. 
\end{proof} 

\subsection{Sequence of Reductions}

Now we are ready to complete the proof, 
doing the reduction prescribed in  Theorem~\ref{DAHA-Q}
(for orientation of all edges towards the node)
step by step according to Proposition~\ref{seqred}.
Note that separating the node and the ``legs'' of our graph we have
$\gl(\alpha^{CM}) \cong \gl_{n\ell} \oplus \gl(\alpha^{CM})'$, where
$$\gl(\alpha^{CM})' \cong \bigoplus_{i=1}^m \g_i,
\qquad \g_i = \bigoplus_{j=1}^{d_i-1} \gl_{\alpha^{CM}_{(i,j)}}.
$$

\begin{prop}
We have in the notation of the Main Theorem a $\gl_{n\ell}$-equivariant isomorphism
$$\qr\left(\D(\rep_{\alpha^{CM}}(Q_{CM})), \gl(\alpha^{CM})', d\itw - \chi\tr\right) \cong
U_{\mu_1}^{\p_1}(\gs_{n\ell}) \otimes \dots \otimes U_{\mu_m}^{\p_m}(\gs_{n\ell}).$$
\end{prop}

\begin{proof}
We have
$$\rep_{\alpha^{CM}}(Q_{CM}) \cong \bigoplus_{i=1}^m \rep_{\alpha^{CM}}(Q_{CM})_i,$$
such that $\g_i$ acts only on  $\rep_{\alpha^{CM}}(Q_{CM})_i$. Therefore the left hand side
is isomorphic to
$$\bigotimes_{i=1}^m \qr\left(\D(\rep_{\alpha^{CM}}(Q_{CM}))_i, \gl_i, d\itw - \chi\tr\right).$$
Combining for each factor Theorem~\ref{Q-Gr} 
with Theorem~\ref{Gr-U}, we obtain the statement of the Proposition.
\end{proof}

At last note that $\gl_{n\ell} \cong \C \oplus \gs_{n\ell}$ and that $\C$ acts on the right hand side by
a scalar, so the result of reduction with respect to $\C$ is either zero (not in our case) or the initial algebra.
Applying the reduction with respect to $\gs_{n\ell}$, we obtain the statement of the Main Theorem.

\section{Construction of representations}

Corollary \ref{consrep} together with Proposition \ref{seqred}
immediately provides a way of constructing many finite dimensional 
representations of spherical symplectic reflection algebras of
types $D_4,E_6,E_7,E_8$, by taking $U_i$ to be finite dimensional
representations. Such representations are mostly new, and would
be interesting to study in detail, which we plan to do in the
future. But they defined only for 
a certain discrete set of values of parameters. Below we
construct the finite dimensional representations which are
defined on hyperplanes (of codimension 1) in the space of
parameters. These representations were discovered in \cite{EM}. 

\subsection{Some isomorphisms of enveloping algebra quotients.} 

Let us establish some isomorphisms between different
$U_\mu^\p(\g)$, to be used in the next subsection,
(they are a peculiarity of $\g=\gs_r$). Let $\p_1$ be a
parabolic subalgebra of $\gs_r$ and $\mu_1$ a character of $\p_1$.
Let $m_1, \dots, m_k$ be the sizes of blocks of the Levi
subalgebra $\l_1\subset\p_1$. Let $\p_2$ be the parabolic
subalgebra with the sizes of blocks $m_1, \dots, m_{i+1}, m_i,
\dots, m_k$ (two neighboring blocks transposed). Consider the
permutation $\sigma\in S_r$
%$$\sigma = \left( {{1}\atop{1}}\ \  {{2}\atop{2}} \dots {{m_1+ \dots+m_{i-1}}\atop{m_1+ \dots+m_{i-1}}}\ \ {{m_1+ \dots+m_{i-1}+1}\atop{m_1+ \dots+m_{i-1}+m_i+1}}{{s}\atop{1}}\ \  {{s+1}\atop{s+1}}
%\dots {{m_1+ \dots+m_k}\atop{m_1+ \dots+m_k}}\right),$$
which is the transposition of the $i$-th and the $i+1$-th blocks.

The following proposition is known, but for reader's convenience 
we include a proof (based on quantum reduction) 
in the appendix to this paper.

\begin{prop}\label{prlevi}
We have  $U^{\p_1}_{\mu_1}(\g) \cong U^{\p_2}_{\mu_2}(\g)$ with
$\mu_2=\sigma(\mu_1+\rho)-\rho$, where $\rho$ is the half-sum of
positive roots.
\end{prop}

More generally, let $\p_1$ and $\p_2$ be parabolic subalgebras of $\gs_r$ with
Levi subalgebras $\l_1$ and $\l_2$. By $\rho_1$ and $\rho_2$
denote the half-sums of
positive roots of $\l_1$ and $\l_2$ respectively.
Suppose that $\l_1 \cong \l_2$, then there exists a 
``block-wise'' permutation $\sigma$
such that $\sigma(\rho_1) = \rho_2$. An advantage of this permutation is that
for a weight $\mu_1: \p_1 \to \C$ we have that $\sigma(\mu_1+\rho)-\rho$ is
a weight of $\p_2$. 

\begin{corollary}\label{prlevi1}
We have  $U^{\p_1}_{\mu_1}(\g) \cong U^{\p_2}_{\mu_2}(\g)$ with
$\mu_2=\sigma(\mu_1+\rho)-\rho$.
\end{corollary}

\def \tp {\tilde{\p}}

Now let $\p''(s,r)$ 
be the parabolic subalgebra in $\g={\mathfrak{sl}}_r$
generated by upper triangular matrices and the elements 
and $f_i=E_{i+1,i}$ for $i\ne
\frac{r}{s}-1$ and $\frac{r}{s}$ not dividing $i$. 
Also, denote by $\tp''(s,r)$ the
subalgebra generated by $\p''(s,r)$ and 
$f_{r/s}=E_{\frac{r}{s}+1,\frac{r}{s}}$.

\medskip

\psfig{file=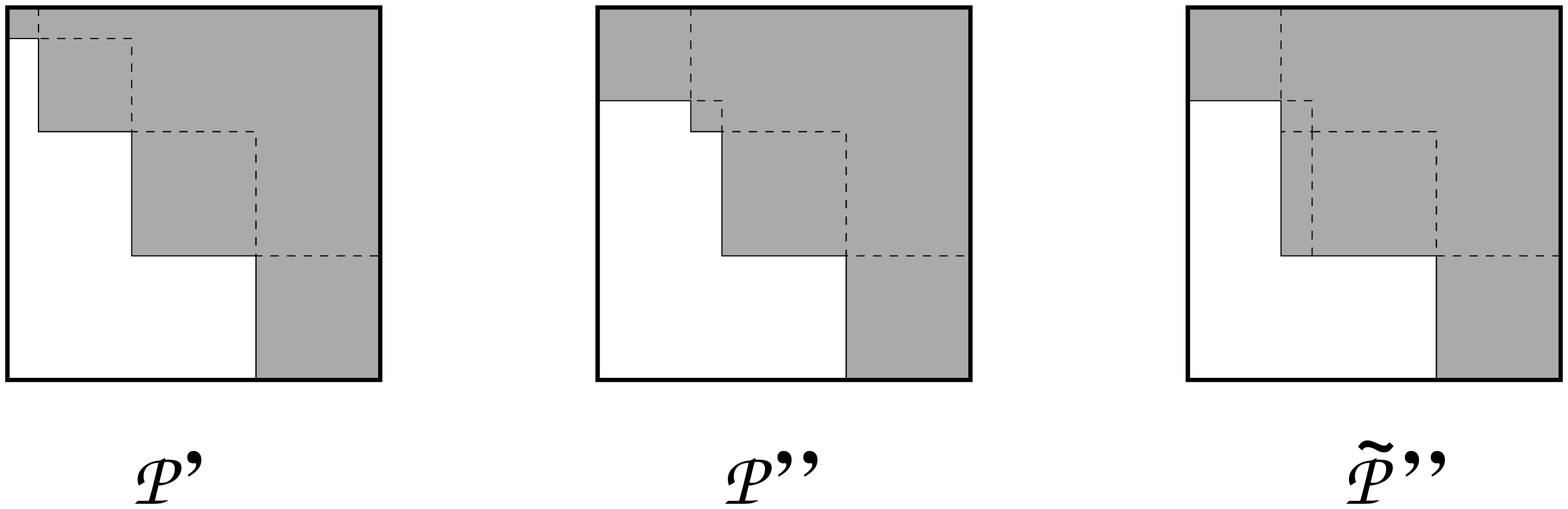,width=12.5cm}

\begin{corollary}\label{prlevi2}
We have $U^{\p''(s,r)}_{\nu}(\g) \cong
U^{\p'(s,r)}_{\mu}(\g)$, where
$$\nu^1 =0, \ \mu^{r/s-1}=0,\ \nu^{r/s-1} =-\mu^1-r/s,\ \
\nu^{r/s} =\mu^1+\mu^{r/s}+r/s-1$$
% \ \ \mbox{and} \ \nu^i = \mu^i \
%\mbox{for all other} \ i.$$ 
and $\nu^i = \mu^i$ for all other $i$. 
\end{corollary}

\begin{proof}
The corollary follows from Proposition~\ref{prlevi}.
\end{proof}

\subsection{The open orbit lemma} 

Let $G=PGL_{n\ell}(\Bbb C)$. 
Let $P_i$ be the parabolic subgroups of $G$ corresponding to 
the Lie algebras ${\mathfrak{p}}_i$, introduced in Theorem
\ref{thmain}. Let $\tilde P''(s,r)$ be the parabolic subgroup of $G$ 
corresponding to the Lie algebra $\tp''(s,r)$. 

\begin{lemma}\label{openorb} 
The group $G$ has an open orbit $X_*$ on the space 
$$
X:=G/P_1\times...\times G/P_{m-1}\times G/\tilde P''(\ell,n\ell),
$$
and the action of $G$ on this orbit is free. 
\end{lemma}

\begin{proof}
The result follows from the fact 
that for any real positive root $\beta$ of a simply laced affine
root system, a generic representation of the corresponding affine quiver 
with dimension vector $\beta$ has one-dimensional endomorphism algebra 
(in particular, is indecomposable), and such a representation is
unique (Kac's theorem, \cite{Ka}). Namely, in our case,
$\beta=n\delta-\alpha_0=(n-1)\delta+\theta$, where $\theta$ is 
the maximal root of the corresponding finite root system. 
So $\beta$ is a real root and we can apply Kac's theorem.
\end{proof} 

\subsection{A representation of
$U_{\mu_1}^{\p_1}\otimes...\otimes U_{\mu_m}^{\p_m}$}

Let $\tilde\p \supset \p$ 
be the parabolic subalgebra of $\g$ generated by $\p$ and $f_\alpha$ for
some simple root $\alpha$. Suppose that $\nu$ is a character of
$\p$, and $\langle\nu,\alpha\rangle\in\ZZ_{\ge0}$.
Then the irreducible highest weight representation 
$L_\nu$ of the Levi subgroup of $\tilde L\subset\tilde P$ is finite dimesional.
The epimorphism $\tilde P\to\tilde L$ 
defines the structure of a $\tilde P$-module 
on $L_\nu$. 

Now let $\nu_m$ be the weight related to $\mu_m$ as in Corollary
\ref{prlevi2}. Assume that $\nu_m^n:=\langle
\nu_m,\alpha_n\rangle$ is a positive integer, and let  
$L:=L_{\nu_m}$ be the corresponding finite dimensional module 
over $\tp_{\ell,n\ell}$. Then $L$ is naturally a representation
of $\hat\p:=\p_1\times...\times \p_{m-1}\times \tilde \p''(\ell,n\ell)$, 
with $\p_i$ acting through the characters $\mu_i$. 

Let $x\in X_*$ be a point in the open orbit, and $x'$ a preimage
of $X$ in $G^m$, and $B$ the formal neighborhood of $x'$ in
$G^m$. Let $Y'$ be the space of regular functions on $B$
with values in $L$, and $Y$ be the space of $f\in Y'$ such that 
for any $z\in \hat\p$, $R_zf=-\pi_L(z)f$, where $R_z$ is the
vector field of right translation by $z$. Thus, by Lemma
\ref{openorb}, $Y$ can be identified with the space of 
$L$-valued regular functions 
on the neighborhood of $1$ in $G$, i.e. $(S\g)^*\otimes L$.   

\begin{prop}\label{repre} The action of $\g^m$ on 
$Y$ by left translations makes it into a representation of 
$U_{\mu_1}^{\p_1}\otimes...\otimes U_{\mu_m}^{\p_m}$.
\end{prop}

\begin{proof}
By construction, $Y$ is a representation of 
$U_{\mu_1}^{\p_1}\otimes...\otimes
U_{\mu_{m-1}}^{\p_{m-1}}\otimes U_{\nu_m}^{\p''(\ell,n\ell)}$.
So the statement follows from Corollary \ref{prlevi2}. 
\end{proof} 

\subsection{Application to $eH_{1,k,c}e$}

Let $\p_i$ be as in Theorem~\ref{thmain}. Fix weights $\mu_i:
\p_i \to \C$, with $\nu_m^n:=\langle \nu_m,\alpha_n\rangle$
being a nonnegative integer $q$
(where $\nu_i$ are related to $\mu_i$ as 
in Corollary \ref{prlevi2}). Let $k,c$ be related to $\mu_i$ as
in Theorem~\ref{thmain}. Let $Y$ be the representation 
constructed in the previous subsection. 

\begin{theorem}\label{genrep} The space
$Y^\g$ (the invariants in $Y$ under the diagonal $\g$-action) 
is a $eH_{1,k,c}e$-module of dimension $\binom{n+q}{n}$. 
\end{theorem}

\begin{proof} 
The fact that $Y^\g$ is a $eH_{1,k,c}e$-module follows from 
Corllary \ref{consrep}. The dimension formula follows from Lemma \ref{openorb}
(which implies that $\g$ acts simply transitively on the formal
neighborhood of $x$). Indeed, the lemma implies that $Y^\g$ can
be identified with $L$, and the dimension of $L$ is 
 $\binom{n+q}{n}$, since $L=S^q\C^{n+1}$. 
\end{proof} 

\begin{remark}
Let $x_i$ be the projection of $x$ to the i-th factor of the
product $G/P_1\times...\times G/P_{m-1}\times G/\tilde
P''(\ell,n\ell)$, and $H_i$ be the space of $L^*$-valued distributions 
on $G$ concentrated on the preimage of $x_i$, satisfying the
condition $R_z\psi=-\pi_{L^*}(z)\psi$, $z\in \p_i$. 
Then $Y=(H_1\otimes...\otimes H_m)^*$, hence 
$Y^\g=((H_1\otimes...\otimes H_m)_\g)^*$. 
Note that $H_i$ are generalized Verma modules for 
appropriate polarizations of $\g$.  
\end{remark}

Note that for generic parameters $k$ and $c$ there are no
finite-dimensional representations of $eH_{1,k,c}e$, they appear
on hyperplanes (see \cite{EM}). Theorem~\ref{genrep} constructs
them for a generic set of parameters at hyperplanes
$$\nu_m^n = \lambda(c)_\o + \frac{k(n-1)}2-1 \in \ZZ_+.$$

\section{Appendix: proof of Proposition~\ref{prlevi}}

First, we show this in the case of $2$ blocks of sizes $m_1, m_2$.
In this case, the space of characters $\p_1\to\C$ is
one-dimensional, hence we can assume that $\mu_1$ is a complex
number such that
$e_{ij}\mapsto\frac{m_2\mu_1}{m_1+m_2}\delta_{ij}$ for
$i=1,\dots,m_1$ and
$e_{ij}\mapsto\frac{-m_1\mu_1}{m_1+m_2}\delta_{ij}$ for
$i=m_1+1,\dots,m_1+m_2$. We have
$$\sigma = \left( {{1}\atop{m_2+1}}\ \  {{2}\atop{m_2+2}} \dots {{m_1}\atop{m_2+m_1}}\ \ {{m_1+1}\atop{1}}\ \dots  {{m_1+m_2}\atop{m_2}}
\right),$$ and $\mu_2=\sigma(\mu_1+\rho)-\rho=-\mu_1-m_1-m_2$.

Let $V$ be the tautological representation of
$\gs_r=\gs_{m_1+m_2}$. By Theorem~\ref{Gr-U}, we have
$U^{\p_1}_{\mu_1}(\g) \cong \mathcal{D}_{\mu_1}\left(\fl (m_1, m_1+m_2)\right)$, 
where 
$$\fl (m_1, m_1+m_2)=\inc(\C^{m_1},V)/GL_{m_1}$$ 
is the Grassmann variety. Due to Theorem \ref{Q-Gr} we have
$$\mathcal{D}_{\mu_1}\left(\fl (m_1, m_1+m_2)\right) =\qr(\D(\C^{m_1}\otimes V),
\gl_{m_1}, d\phi-\mu_1\tr),$$
 where $\gl_{m_1}$ acts on the first
tensor factor, and the map
$d\phi-\mu_1\tr$  is defined as
follows. Let $v_{i,k},\ i=1,\dots,m_1,\ k=1,\dots,m_1+m_2,$ be a
basis of $\C^{m_1}\otimes V$, and $v^*_{i,k}$ be the dual basis of
$\C^{m_1}\otimes V^*$, then
$$(d\phi-\mu_1\tr)(e_{ij})=\sum\limits_{k=1}^{m_1+m_2}v_{j,k}\partial_{v_{i,k}}-\mu_1\delta_{i,j}
\quad : \quad \gl_{m_1}\to\D(\C^{m_1}\otimes V).
$$

Let $\fou:\D(\C^{m_1}\otimes V)\tilde\to\D(\C^{m_1}\otimes V^*)$ be
the \emph{Fourier transform}, which sends $v_{i,k}$ to
$\partial_{v^*_{i,k}}$, and $\partial_{v_{i,k}}$ to $-v^*_{i,k}$.
We have
\begin{multline*}\fou \left((d\phi-\mu_1\tr)(e_{ij})\right )=
\fou\left(\sum\limits_{k=1}^{m_1+m_2}v_{j,k}\partial_{v_{i,k}}-\mu_1\delta_{i,j}\right)=\\
=\sum\limits_{k=1}^{m_1+m_2}-\partial_{v^*_{i,k}}v^*_{j,k}-\mu_1\delta_{i,j}=
\sum\limits_{k=1}^{m_1+m_2}-v^*_{j,k}\partial_{v^*_{i,k}}-(\mu_1+m_1+m_2)\delta_{i,j}=\\=-(d\phi-\mu_2\tr)(e_{ji}).\end{multline*}
Thus, the Fourier transform gives an isomorphism
\begin{multline*}\qr(\D(\C^{m_1}\otimes V), \gl_{m_1},
d\phi-\mu_1\tr)\cong\qr(\D(\C^{m_1}\otimes V^*), \gl_{m_1},
d\phi-\mu_2\tr)=\\=\mathcal{D}_{\mu_2}(\inc(\C^{m_1},V^*)/GL_{m_1}).\end{multline*}

Note that $\inc(\C^{m_1},V^*)/GL_{m_1}$ is naturally identified
with the space of surjective operators $V^*\to\C^{m_2}$ up to the
$GL_{m_2}$-action on the target space, and the latter is naturally
$\inc(\C^{m_2},V)/GL_{m_2}$ (this is an isomorphism of algebraic
varieties). Thus we have
\begin{multline*}U^{\p_1}_{\mu_1}(\g)=\qr(\D(\C^{m_1}\otimes V), \gl_{m_1},
d\phi-\mu_1\tr)\cong\\\cong
\mathcal{D}_{\mu_2}(\inc(\C^{m_2},V)/GL_{m_2})=\mathcal{D}_{\mu_2}(\fl (m_2,
m_1+m_2))=U^{\p_2}_{\mu_2}(\g).\end{multline*}

Note that all the isomorphisms above preserve the action of
$\gl_r=\gl_{m_1+m_2}$. This enables us to pass to the general case
as follows. Consider the parabolic subalgebra
$\p\subset\g$ generated by the subalgebra of upper-triangular
matrices and the Levi subalgebra of block-diagonal matrices with
the sizes of blocks equal to $m_1,\dots, m_{i-1}, m_i+m_{i+1},
\dots,m_k$. Let $P\subset G$ be the corresponding parabolic
subgroup. The variety 
$$SL_r/P_1=\fl (m_1, m_1+m_2, \dots,
m_1+\dots +m_{k-1}; r)$$ 
can be regarded as the homogeneous bundle
$SL_r*_P\fl (m_i ; m_i+m_{i+1})$, where $P$ acts on $\fl (m_i ;
m_i+m_{i+1})$ through the homomorphism $P\to GL_{m_i+m_{i+1}}$.
Respectively,  we have 
$$SL_r/P_2=
%\fl (m_1, m_1+m_2, \dots,m_1+\dots
%+m_{i-1}+m_{i+1}, m_1+\dots +m_{i-1}+m_{i+1}+m_i,\dots, m_1+\dots
%+m_{k-1}; r)=
SL_r*_P\fl (m_{i+1}; m_i+m_{i+1}).$$

Consider the homogeneous bundle $SL_r*_{[P,P]}(\C^{m_i}\otimes
V)$, where $V$ is the tautological representation of
$SL_{m_i+m_{i+1}}$, and $P$ acts on $V$ through the homomorphism
$[P,P]\to SL_{m_i+m_{i+1}}$. Suppose
$$\mu_1=\sum\limits_{s=1}^{k-1}\mu_1^{(s)}\tr_s \quad  \mbox{where} \quad
\tr_{s}(p)=\tr\left(
p\sum\limits_{i=m_1+\dots+m_{s-1}+1}^{m_1+\dots+m_s}e_{ii}\right).$$
 Then
the algebra $U^{\p_1}_{\mu_1}(\g)=\mathcal{D}_{\mu_1}(SL_r*_P\fl (m_i ;
m_i+m_{i+1}))$ is the following quantum Hamiltonian reduction:
$$\qr(\D(SL_r*_{[P,P]}(\C^{m_i}\otimes V)),
\gl_{m_i}\oplus\p/[\p,\p],
(d\phi-(\mu_1^{(i)}-\mu_1^{(i+1)})\tr)\oplus(\E-\bar\mu_1^{(i)})),$$
where $\bar\mu_1^{(i)}=\sum\limits_{s\ne
i,i+1}\mu_1^{(s)}\tr_s+\frac{\mu_1^{(i)}m_i+\mu_1^{(i+1)}m_{i+1}}{m_i+m_{i+1}}(\tr_i+\tr_{i+1})$.
Here $GL_{m_i}$ acts on $\C^{m_i}$, and $P/[P,P]$ acts on the
homogeneous bundle by right shifts. Note that
$\bar\mu_1^{(i)}=\bar\mu_2^{(i)}$ and
$\mu_2^{(i)}-\mu_2^{(i+1)}=-(\mu_1^{(i)}-\mu_1^{(i+1)})-m_i-m_{i+1}$.

Since the Fourier transform $\fou: \D(\C^{m_i}\otimes
V)\tilde\to\D(\C^{m_i}\otimes V^*)$ is $\gl_{m_1+m_2}$-invariant,
there is a well-defined \emph{fiberwise Fourier transform}
$$\fou_P:\D(SL_r*_{[P,P]}(\C^{m_i}\otimes V))\tilde\to
\D(SL_r*_{[P,P]}(\C^{m_i}\otimes V^*)).$$ The computation in the
case of two blocks shows that this isomorphism gives an
isomorphism of quantum Hamiltonian reductions
\begin{multline*}\qr(\D(SL_r*_{[P,P]}(\C^{m_i}\otimes V)),
\gl_{m_i}\oplus\p/[\p,\p],
(d\phi-\mu_1^{(i)}\tr)\oplus(\E-\bar\mu_1^{(i)}))\cong\\\cong\qr(\D(SL_r*_{[P,P]}(\C^{m_i}\otimes
V^*)), \gl_{m_i}\oplus\p/[\p,\p],
(d\phi-\mu_2^{(i)}\tr)\oplus(\E-\bar\mu_1^{(i)})).\end{multline*}
The right-hand side is naturally isomorphic to \begin{multline*}
\qr(\D(SL_r*_{[P,P]}(\C^{m_{i+1}}\otimes V)),
\gl_{m_{i+1}}\oplus\p/[\p,\p],
(d\phi-\mu_2^{(i)}\tr)\oplus(\E-\bar\mu_2^{(i)}))=\\=\mathcal{D}_{\mu_2}(SL_r*_P\fl
(m_{i+1} ; m_i+m_{i+1}))=U^{\p_2}_{\mu_2}(\g). \end{multline*}

Proposition~\ref{prlevi} is proved.

\end{document}